\documentclass[reqno,11pt]{amsart}
\usepackage{amsmath,amsthm}
\usepackage{amssymb}
\usepackage{xypic}
\usepackage{mathrsfs}
\usepackage{hyperref}
\usepackage{graphicx}
\usepackage{perpage}
\usepackage{extarrows}
\MakePerPage{footnote}
\usepackage{setspace}
\date{\today}
          
 \usepackage[usenames,dvipsnames]{pstricks}
\usepackage[shell]{pdftricks}
 \begin{psinputs}
 \usepackage{pst-grad} 
 \usepackage{pst-plot} 
 \usepackage{epsfig}
 \end{psinputs}

\title{Transversality of smooth definable maps in O-minimal structures}

\author{Nhan Nguyen and Saurabh Trivedi}

\address{Departamento de Matem\'atica, ICMC - Universidade de S\~ao Paulo, 134560-970, S\~ao Carlos, S.P, Brasil}
\email{nguyenxuanvietnhan@gmail.com}
\email{saurabh.trivedi@gmail.com}

\begin{document}
\newtheorem{thm}{Theorem}[section]
\newtheorem{lem}[thm]{Lemma}
\theoremstyle{theorem}
\newtheorem{prop}[thm]{Proposition}
\newtheorem{defn}[thm]{Definition}
\newtheorem{cor}[thm]{Corollary}
\newtheorem{example}[thm]{Example}
\newtheorem{xca}[thm]{Exercise}

\newcommand{\bb}{\mathbb}
\newcommand{\al}{\mathcal}
\newcommand{\ak}{\mathfrak}
\newcommand{\fs}{\mathscr}

\theoremstyle{remark}
\newtheorem{rem}[thm]{Remark}
\parskip .12cm

\maketitle

\begin{abstract} 
We present a  definable smooth version of the Thom transversality theorem. We show further that  the set of non-transverse definable smooth maps is nowhere dense in the definable smooth topology.  Finally, we prove a definable version of a theorem of Trotman which says that the Whitney $(a)$-regularity of a stratification is necessary and sufficient for the stability of transversality.
\end{abstract}

\section{Introduction}
In o-minimal structures over real closed fields, the topology on the set of definable $\al C^p$-maps between definable $\al C^p$-manifolds induced from the Whitney strong topology is too strong for a transversality theorem to hold (see Example \ref{example1}). In other words, transversality of definable maps is not a generic condition for the induced topology. For this reason Shiota \cite{shiota1} introduced a weaker topology for the semialgebraic structure that can also be adapted to the wider setting of o-minimal structures.  We call this the definable topology. In this topology  Loi \cite{loi1} proved an analogue of the Thom transversality theorem in o-minimal structures. 

The result of Loi was just proved for definable maps of class $\al C^p$ for $p \in \bb N$. He also mentions in his article that it is  not clear whether the theorem holds for $p = \infty$ or $p = \omega$. In this article we deal with the smooth case. We will restrict to exponential o-minimal structures for which the smooth cell decomposition holds. The reason for this is two-fold: first the infinite differentiability is not well defined in o-minimal structures in general as proved by Wilkie \cite{wilkie1} and second we use an approximation result of Fischer \cite{fischer1} and smooth definable bump functions in our proofs. Smooth definable bump functions do not exist in polynomially bounded o-minimal structures as a consequence of a result of Miller \cite{miller1}, and not every o-minimal structure admits smooth cell decomposition (see Le Gal and Rolin \cite{lr}). 

The first result in this article is a definable smooth version of the Thom transversality theorem (Theorem \ref{thm_transversality}). To prove the result we  give a simplified definition of the definable topology on the set of smooth definable maps using definable jet spaces (Subsection 2.3). This definition is equivalent to the one used in the articles of Escribano \cite{escribano1}, Fischer \cite{fischer1}, Shiota \cite{shiota1} and Loi \cite{loi1}. The advantage of our definition is that certain properties, namely Lemmas \ref{lem_jet_map1}, \ref{lem_jet_map2}, \ref{lem_proper} and \ref{lem_restriction}, of the topology follow easily and these are not so apparent from the other definition. Our result then follows by combining the approximation result of Fischer \cite{fischer1} with  the method of Loi \cite{loi1}.

We next discuss the stability properties of transversality to stratifications. In the smooth case the most basic result about stability is that the set of smooth maps transverse to a closed submanifold is open in the strong topology. A more general result can be stated as follows: let $M$, $N$ be smooth manifolds and $\Sigma$ be a stratification of a closed subset in $N$. If the dimension of $M$ is big enough then the following are equivalent:

(1) $\Sigma$ is Whitney $(a)$-regular;

(2) the set $\{f \in \al C^{\infty}(M,N) : f \pitchfork \Sigma\}$ is open in the strong topology.

That (1) implies (2) is a result of Feldman \cite{Feldman}. In the o-minimal setting, an analogue of Feldman's result was included in the transversality theorem of Loi \cite{loi1}. It is easy to show that the same also holds for smooth definable maps. Then, using  the fact that any stratification of a definable set admits a Whitney $(a)$-regular refinement, we show that the set of definable smooth maps non-transverse to a given finite collection of definable $\al C^1$-manifolds is nowhere dense in the definable topology (Theorem \ref{cor_strong_transversality}). To our knowledge such a result has not been observed before in the smooth or complex settings and might not even be true. It is worth mentioning here that Mather \cite{Mather} has also sought such a result for generic projections. He asks whether projections transverse to Boardman manifolds contain an open and dense set; see the comment after Theorem 2 of Section 5 in his article.

The converse of Feldman's result, i.e. (2) implies (1) was proved by Trotman \cite{trotman1} (see also \cite{st}). The main difficulty in proving the result of Trotman is the existence of a smooth map transverse to a stratification with a given derivative at a point. Trotman used two non-trivial results to prove the existence of such a map. First the Baire property of the Whitney strong topology and second the density of transverse maps. Since it is not clear whether the definable topology has the Baire property the arguments of Trotman do not work in the definable case. To overcome this difficulty we provide a new more constructive method (see Lemma \ref{givendf}) that avoids the use of the Baire property to show that a definable version of Trotman's result holds (Theorem \ref{thm_openness_a_regularity}). This new method might have other applications.

Thoughout the paper, we work with o-minimal structures defined on the ordered field of real numbers $(\bb R, <, +, ., 0,1)$. By a definable set (resp. a definable map) we mean a set (resp. a map) which is definable in the given o-minimal structure. We refer the reader to Coste \cite{coste1} and van den Dries \cite{dries} for definitions and basic results on o-minimal structures. 

Let us fix some notations used in the article. A set $M \subset \bb R^m$ is called \emph{a definable $\al C^p$-submanifold} ($0\leq p \leq \infty$) if $M$ is a definable set and is a $\al C^p$-submanifold of $\bb R^m$. 
By a definable $\al C^p$-manifold we mean a definable $\al C^p$-submanifold of some Euclidean space.  If $p = \infty$ it is called a definable smooth manifold.

By a definable stratification of a given definable set $V \subset \bb R^m$ we mean a partition $\Sigma$ of $V$ into finitely many definable $\al C^1$-submanifolds of $\bb R^m$, called strata, such that for any $X, Y$  in $\Sigma$ such that $Y \cap \overline{X} \neq \emptyset$, we have $Y\subset \overline{X}$ where $\overline{X}$ denotes the closure of $X$. We call  $(X, Y)$ a \emph{pair of adjacent strata}. A stratum $X$ is said to be \emph{Whitney $(a)$-regular} over $Y$ if for every point $y \in Y \cap \overline{X}$ and for any sequence of points $\{x_n\}$ in $X$ tending to $y$ such that the sequence of tangent spaces $\{T_{x_n}X\}$ tends to $\tau \in \bb G_m^{\dim X}$, we have $T_y Y \subset \tau$. A stratification is called Whitney $(a)$-regular if for every pair of adjacent strata $(X,Y)$, $X$ is Whitney $(a)$-regular over $Y$. We denote by $\bb R$ the set of the real numbers, $\bb N= \{0, 1, \ldots\}$ the set of the natural numbers and $\bb G_m^k$ the Grassmannian of all $k$-linear subspaces of $\bb R^m$.\\

\noindent 
\textbf{Acknowledgements.} We would like to thank Professor Maria Ruas for her interest and useful discussion on the paper. We also thank Professor David Trotman for his valuable comments on the manuscript. The first author was supported by FAPESP grant 2016/14330-0 and the second author was supported by FAPESP grant 2015/12667-5.

\section{Topology on  definable maps}
In this section we give a definition of a topology on the set of definable smooth maps using definable jet spaces. This definition is equivalent to the one used by Escribano \cite{escribano1}, Fischer \cite{fischer1}, Loi \cite{loi1} and Shiota \cite{shiota1}. 

\subsection{Definable tubular neighborhood}
Let $\al D$ be an o-minimal structure and $M \subset \bb R^m$ be a definable $C^p$-submanifold with ($1\leq p \leq \infty$). We denote by $\al NM$  the normal bundle of $M$, i.e. is the set $\{(x, v) \in M \times \bb R^m \text{ such that } v \perp T_x M\}$. Then $\al NM$ is a definable $\al C^{p-1}$-submanifold of $\bb R^m \times \bb R^m$. Consider the definable $\al C^{p-1}$-map $\varphi : \al NM \to \bb R^m, (x, v) \mapsto (x + v)$. Following the arguments of Coste \cite{coste1} we can show that there exists a definable open neighborhood $U$ of $M\times \{0\}$ in $\al NM$ such that the restriction $\varphi|_U$ is a definable $\al C^{p-1}$-diffeomorphism onto an open neighborhood $T_M$ of $M$ in $\bb R^m$. We can assume that $U$ is of the form 
$$\{(x, v) \in \al NM, \|v\| < \varepsilon(x)\}$$
where $\varepsilon: M \to (0,\infty)$ is a definable $\al C^k$-function with $k\in \bb N$. 

Let $\pi_1: \al NM \to M$ defined by $\pi_i(x, v) = x$  be  the projection onto $M$. Then the map 
$$ \pi_M: T_M \to M, w \mapsto \pi_1(\varphi^{-1} (w))$$ is a definable $\al C^{p-1}$-retraction. We call the pair $(T_M, \pi_M)$ a \emph{definable tubular neighborhood} of $M$ in $\bb R^m$.

\subsection{Definable jet space}

Let $M\subset \bb R^m$ and $N\subset \bb R^n$ be definable $\al C^p$-submanifolds ($1\leq p \leq \infty$). Denote by $\al D^p(M,N)$ the set of all definable $C^p$-maps between $M$ and $N$ and for an integer $k < p$ denote by $J^k(M, N)$ the $k$-jet bundle of smooth maps from $M$ to $N$. Define 
$$J^k_{\al D}(M, N):= \{ j^k f(x) \in J^k(M, N):  x\in M \text{ and } f \in \al D^{p} (M, N)\}\}$$
and call it the \emph{definable $k$-jet space}.

It is possible to embed $J^k_{\al D}(M, N)$ into a Euclidean space as follows: let us denote by $P^k(\bb R^m)$ the set of all polynomials in $m$ variables with degree at most $k$ and the constant term zero. Let $A$ be the cardinality of the set $\{\alpha = (\alpha_1, \ldots, \alpha_m) \in \bb N^m: 1\leq |\alpha|:=\alpha_1 +\ldots +\alpha_m \leq k\}$. We may identify $P^k(\bb R^m)$ with $\bb R^A$ by 
$$ \sum_{1\leq \alpha \leq k} a_\alpha X^\alpha \leftrightarrow (a_\alpha)_{1\leq \alpha \leq k}.$$

Suppose that $M \subset \bb R^m$ and $N \subset \bb R^n$ are open subsets. Recall that the $k$-jet of a map $f \in \al D^p(M,N)$ at a point $x \in M$ is the truncated Taylor polynomial of degree $k$. That is,
$$j^kf(x)(X) = f(x) +  \sum_{1\leq \alpha \leq k} \frac{\partial^\alpha f(x)}{\alpha!} (X -x)^\alpha \leftrightarrow (x, f(x), \partial^\alpha f(x))_{1\leq \alpha \leq k}$$
where  $$\partial^\alpha f(x):= \frac{\partial^{|\alpha|} f(x)}{\partial x_1^{\alpha_1} \ldots \partial x_m^{\alpha_m}}.$$
Since all polynomial maps are definable, we have $J_{\al D}^k(M, N) = J^k(M, N)\equiv M\times N \times \bb R^{n A}$.

In general, for any definable submanifolds $M \subset \bb R^m$ and $N \subset \bb R^n$ , we take $(T_M, \pi_M)$ and $(T_N, \pi_N)$ to be definable tubular neighborhoods of $M$ and $N$ in $\bb R^m$ and $\bb R^n$ respectively and identify 
$$
J_{\al D}^k(M, N) \equiv \{j^k (f \circ \pi_M) (x): f\in \al D^{\infty}(M, N), x \in M\} \subset J_{\al D}^k (T_M, T_N).$$
Now, consider $\pi_M$ as a  map from $T_M$ to $\bb R^m$ and  $\pi_N$ as a map from $T_N$ to $\bb R^n$. Then we have
$$J_{\al D}^k(M, N) \equiv  \{j^k (\pi_N \circ g \circ \pi_M) (x): g\in \al D^{\infty}(T_M, T_N), x \in M\}.$$
The composition $g\circ \pi_M(x)$ for $x\in M$ makes sense because $M\subset T_M$. Note that $J_{\al D}^k(M, N)$ is a  definable $\al C^{p-k}$-submanifold of $T_M \times T_N \times \bb R^{nA}$ (see \cite{loi1},  \cite{shiota2}). 

\subsection{Definable topology}
Let $M\subset \bb R^m$ and $N\subset \bb R^n$ be definable $\al C^p$-submanifolds ($1\leq p \leq \infty$). Let $k < p$ and  $U$ be a definable subset of $J_{\al D}^k(M, N)$. Set 
$$\al M(U) := \{f\in \al D^p(M, N): j^k f(M) \subset U\}.$$
We can define a topology on $\al D^p(M,N)$ by regarding the family $\{\al M(U)\}$ as its basis. We then call this topology the \emph{$\al D^k$-topology}.

Since $J_{\al D}^k(M, N)$ is a subset of $\bb R^{m +n + nA}$ we will endow $J_{\al D}^k(M, N)$ with the metric induced from $\bb R^{m +n + nA}$. For convenience we take the metric $d(x, y) := \| x -y\|_\infty = \max \{|x_1 - y_1|, |x_2 - y_2|,\ldots\}$. In fact the choice of the metric does not affect the topologies because all metrics on a Euclidean space  are equivalent. Now we define
$$\al U^k_\varepsilon(f) := \{g \in\al D^{p}(M, N) : \forall x \in M, d(j^kf(x), j^kg(x))< \varepsilon(x)\}$$
where $f \in \al D^{\infty} (M, N)$ and $\varepsilon: M \to (0,\infty)$ is a continuous definable function. Then the family $\{\al U^k_\varepsilon(f)\}$ also forms a basis of the $\al D^k$-topology.  To see this fact let us consider the map 
$$ \lambda: J^k_\al D(M, N) \to \bb R, j^k g(x) \mapsto \varepsilon(x) -d(j^k f(x), j^kg (x)).$$
Since $\lambda$ is a definable continuous map, $\lambda^{-1}(0,\infty)$ is an open definable set in $J^k_\al D(M, N)$. Put $U:= \lambda^{-1}(0,\infty)$, then $\al U^k_\varepsilon(f) = \al M (U)$.  It remains to show that for any open neighborhood $W$ of $f$ in the $\al D^k$-topology there is a definable continuous function $\theta : M \to (0, \infty)$ such that $\al U_\theta^k(f) \subset W$. Let $V$ be a definable open subset of $W$ such that $f\in V$. Consider the following function
$$ \delta (x) := \inf\{d(j^kg(x)- j^k f(x)): j^kg(x)\in J^k_{\al D}(M, N)\setminus V\}.$$
Note that $\delta$ is a positive definable function. Choosing a definable continuous function $\theta : X \to (0,\infty)$ such that $\theta < \delta$ we have   
$\al U^k_\theta (f) \subset V$.

Now suppose that $M, N$ are definable smooth manifolds. We define the $\al D^\infty$-topology on $\al D^\infty(M, N)$ as follows. Denote by $W^k$ the basis of the $\al D^k$-topology on $\al D^\infty(M, N)$ and set
$$ W^\infty := \bigcup_k W^k.$$ 
The $\al D^\infty$-topology on $\al D^\infty(M, N)$ is the topology whose basis is $W^\infty$. 

\begin{rem} The definition of the definable topology on $\al D^p(M, N)$ using vector fields as used in \cite{escribano1}, \cite{fischer1}, \cite{loi1} and  \cite{shiota1} is as follows: 
	
	For each $k < p$,  let us fix $V_1,\ldots,V_r$ definable $\al C^k$-vector fields on $M$ such that for each $x \in M$, $V_1(x),\ldots,V_r(x)$ generate the tangent space $T_xM$. For $f \in \al D^{p}(M,N)$ and $\varepsilon : M \to (0,\infty)$, a continuous definable function, we define $$\al B^0_{\varepsilon}(f) := \{g \in \al D^{p}(M,\bb R) : |f(x) - g(x)| < \varepsilon(x)\}$$
	and 
	\begin{align*}
	\al B^k_{\varepsilon}(f):=& \{g \in \al D^{p}(M,\bb R) : |V_{i_1}\ldots V_{i_j} (f-g)(x)| < \varepsilon(x) \\ & \quad \forall x \in M, 1 \leq i_1,\ldots,i_j \leq r, 1 \leq j \leq k\}  
	\end{align*}
	for $k>0$. 
	
The $\al D^k$-topology on $\al D^{p}(M,\bb R)$ is the topology whose basis is given by the sets of the form $\al B^0_{\varepsilon}f$ and  $\al B^k_{\varepsilon}f$. The $\al D^k$-topology on $\al D^{p}(M,\bb R^n) = \al D^{p}(M,\bb R) \times \ldots \times \al D^{p}(M,\bb R)$ is defined to be the product topology. For $N$ a definable smooth submanifold of $\bb R^n$, we consider $\al D^{p}(M,\bb R^n)$ as a subset of  $\al D^{p}(M,\bb R^n)$ and define the $\al D^k$-topology on $\al D^{p}(M,N)$  to be the topology induced from $\al D^{p}(M,\bb R^n)$. 

If $M, N$ are definable smooth manifolds then the $\al D^{\infty}$-topology on $\al D^{\infty}(M,N)$ is the topology whose basis is given by the union of all open sets of $\al D^k$-topologies on $\al D^{\infty}(M,N)$ for $k = 0,\ldots$. 
	
Notice that the definition of the topologies above does not depend on the choices of vector fields (see \cite{escribano1}). It is not difficult to prove that the topologies defined by using jet space and the one by using vector fields are the same. A hint for this is to take $V_1, \ldots, V_n$ the orthogonal projections of the canonical vector fields of $\bb R^m$ onto the tangent space of $M$.
\end{rem}

\begin{rem}
Let $M, N$ be $\al C^p$-manifolds. Recall that the $\al C^k$-topology ($k \leq p$) on $\al C^{p}(M,N)$ is the topology generated by the basis given by the sets of the form $\{f \in \al C^{p}(M,N) : j^kf(M) \subset U\}$ where $U$ is an open set in the $k$-jet bundle $J^k(M,N)$. Recall also that if $\varepsilon : M \to (0, \infty)$ is a continuous function, $f \in \al C^{p}(M,N)$ and $d$ is a metric on $J^k(M,N)$, then the set 
$$B_{\varepsilon}f = \{g \in \al C^{p}(M,N) : d(j^kf(x),j^kg(x)) < \varepsilon(x)\,\,\, \forall x \in M\}$$ is an open neighbourhood of $f$ in the $\al C^k$-topology. Then, it is well known that a sequence $f_i$ converges to $f$ in the $\al C^k$-topology if there exists a compact set $K$ on which $f_i$ converges uniformly to $f$ and all but finitely many $f_i$'s are equal to $f$ outside $K$ (see  \cite{gg}).

Now suppose that $M$ and $N$ are two definable $\al C^p$ manifolds in some o-minimal structure. Consider $\al D^p(M,N)$ as a subset of $\al C^p(M,N)$. We remark that for the topology on  $\al D^p(M,N)$ induced by the $\al C^k$-topology on $\al C^{p}(M,N)$, transversality to submanifolds in the target is not a generic condition. This is quite easy to see in the semialgebraic setting:

\begin{example}\label{example1}\rm
Let $M = \bb R$, $N = \bb R^2$, $S \subset N$ be the $x$-axis and $f : M \to N$ be defined by $f(x) = (x,0)$. Clearly $f$ is semialgebraic and is non-transverse to $S$. To prove that transversality is not a generic condition for the induced topology on $\al D^p(M,N)$ of semialgebraic $\al C^p$-maps, we exhibit a neighbourhood of $f$ that does not contain any transverse map. For this, consider the continuous function $\varepsilon : M \to \bb R^+$ given by $\varepsilon(x) = e^{-x}$. Since the structure of semialgebraic sets is polynomially bounded, for any semialgebraic map $g$ such that $\|f(x)-g(x)\| < \varepsilon(x)$, there exists $a \in \bb R$ such that $f(x) = g(x)$ for all $x > a$. Thus no semialgebraic map in the $\varepsilon$-neighborhood of $f$ can be transverse to $S$. 
\end{example}

It is worth noticing that a sequence $\{f_i\}$ converging to $f \in \al D^p(M, N)$ in the $\al D^k$-topology need not be equal outside a compact set as in the case of $\al C^p(M, N)$ (with the $\al C^k$-topology). Again, it is easy to realize this fact in the semialgebraic case with $k=0$.

\begin{example}\rm
Consider $M = \bb R$, $N = \bb R^2$ as semialgebraic $\al C^1$-manifolds. Let $f, f_i: M \to N$ be semialgebraic $\al C^1$-maps defined by $f(x) = (x,0)$,  and
$f_i(x) = (x, x^{-2i})$ outside the compact set $V = [-1, 1]$ and inside $V$ $\{f_i\}$ uniformly converging to $f$. Since the semialgebraic structure is polynomially bounded, it is not hard to show that for any semialgebraic continuous function $\varepsilon: M \to \bb R^+$, we have $\max\{\|(f_i (x)- f(x)\| \}< \varepsilon(x), \forall x \in M$ for $i$ big enough. This means that $\{f_i\}$ converges to $f$  in the $\al D^0$-topology but the $f_i$ are not equal to $f$ outside a compact set.  
\end{example}
\end{rem}

From now on we just restrict to the $\al D^\infty$-topology. It can be easily checked that Propositions 3.4, 3.5, 3.6, 3.9 and 3.10 in Chapter II of Golubitsky and Guillemin \cite{gg} still hold for the $\al D^\infty$-topology in any o-minimal structure. For the sake of clarity we list some of them here and will use in the later sections. 

\begin{lem}\label{lem_jet_map1}
	Let $M, N$ be  definable smooth manifolds. The map 
	$$j^k : \al D^\infty (M, N) \to \al D^\infty (M, J_{\al D}^k(M, N)), f \mapsto j^kf$$ is continuous in the $\al D^\infty$-topology. 
\end{lem}
\begin{proof}
	See Proposition 3.4 in \cite{gg}.
\end{proof}

\begin{lem}\label{lem_jet_map2}
	Let $M $, $N$ and $P $ be definable smooth manifolds. Let $f: N \to P$ be a definable smooth map. The the map 
	$f_* : \al D^\infty(M, N) \to \al D^\infty(M, P)$ given by
	$g\mapsto f \circ g$ is a continuous map in the $\al D^\infty$-topology.  
\end{lem}
\begin{proof}
	See Proposition 3.5 in \cite{gg}. 
\end{proof}

\begin{lem}\label{lem_proper}
	Let $M, N, P$ be smooth definable manifolds. Let $f: P \to M$ be a definable smooth map. If $f$ is proper then the map
	$$f^*: \al D^\infty(M, N) \to \al D^\infty(P, N), g\mapsto g\circ f$$ is continuous in the $\al D^\infty$-topology.
\end{lem}
\begin{proof}
	See note 2, page 49, \cite{gg}.
\end{proof}

\begin{lem}\label{lem_restriction}
	Let $M \subset \bb R^m, N \subset \bb R^n$ be definable smooth manifolds. Let $(T_M,\pi_M)$ be a definable smooth tubular neighborhood of $M$ in $\bb R^m$. Then, the restriction map 
	$$\al D^\infty(T_M, N) \to \al D^\infty(M,  N), f \mapsto f|_M$$
	is continuous in the $\al D^\infty$-topology. 
\end{lem}

\begin{proof} Consider the following map
	$$\iota^*: \al D^\infty(T_M, N) \to \al D^\infty(M, N), f \mapsto f\circ \iota = f|_M$$ 
	which is induced by the inclusion map $\iota: M \to T_M$.
	Since $M$ is closed in $T_M$, $\iota$ is proper. By Lemma \ref{lem_proper}, $\iota^*$ is continuous.
\end{proof}

\begin{rem} In general if $f$ is not proper the map $f^*$ in Lemma \ref{lem_proper} is not continuous.  For instance, consider an o-minimal structure that contains the graph of the exponential function. Suppose $M = N= \bb R$, $P = (0,1) \subset M$ and $f: P \to M$ be the inclusion map. Clearly $f$ is not a proper map. We will show that $f^*$ is not continuous. Let $\varepsilon: P \to (0,\infty)$ be a definable continuous map such that $\lim_{x\to 0} \varepsilon (x) = 0$. First, note that the set  $$ U_P: = \{ g\in \al D^\infty(P, N): |g (x)| < \varepsilon\}$$
is an open set in $\al D^\infty (M, N)$ with the $\al D^\infty$-topology. Set $U_M := (f^*)^{-1} (U_P)$. To show that $f^*$ is not continuous, it suffices to show that $U_M$ is not open. It is easy to see that $U_M$ contains the zero function and if $h\in U_M$ then $h(0) =0$. We will prove that there is no basic open neighborhood of the zero function contained in $U_M$. Indeed, if $U$ is an arbitrary basic neighborhood of the zero function, then by definition, there are $k \in \bb N$ and $\delta: M \to \bb (0, \infty)$, a definable continuous function, such that $U =\al U_{\delta}^k (0)$. By Lemma \ref{lem_appro_def}, there is a definable smooth positive function $\varphi:M \to N$ such that $ \varphi \in \al U_{\delta}^k (0)$. Since $\varphi(0) \neq 0$, $\varphi \not\in U_M$. This implies that $U \not \subset U_M$.
\end{rem}

For the rest of the paper, we restrict to an \emph{exponential o-minimal structure} (i.e. a structure containing the graph of the exponential function) which admits smooth cell decomposition. 

\section{Transversality Theorem}
Recall that a map $f : M \to N$ between two $\al C^1$-manifolds is said to be transverse to a $\al C^1$-submanifold $S \subset N$ at $p \in M$ if either $f(p) \notin S$ or $D_pf(T_pM) + T_{f(p)}S = T_p N$ and in that case we write $f \pitchfork_p S$ and $f \pitchfork S$ in case $f$ is transverse at all points of $M$. If $\Sigma$ is a collection of $\al C^1$-submanifolds of $N$, then $f$ is said to be transverse to $\Sigma$, denoted by $f \pitchfork \Sigma$, if $f$ is transverse to each member of $\Sigma$. 

This section is devoted to the proof of the following theorem.
 
\begin{thm}\label{thm_transversality}  Let $M$ and $ N $ be definable smooth manifolds.  Let $\Sigma$ be a finite collection of definable $\al C^1$-manifolds of $J_{\al D}^k(M, N)$. Then, for any $k \in \bb N$
	
(i) the set $\al A^k_{ \Sigma}:=\{ f \in \al D^\infty (M, N): j^k f \pitchfork \Sigma\}$ is dense in $\al D^\infty(M, N)$ with the $\al D^\infty$-topology.
	
(ii) the set $\al A^k_\Sigma$ is open (and dense) in $\al D^\infty(M, N)$ with the $\al D^\infty$-topology if $\,\,\Sigma$ is a Whitney $(a)$-regular stratification of some definable closed subset of $J_{\al D}^k(M, N)$. 
\end{thm}

To prove Theorem \ref{thm_transversality} we need the following lemmas. 

\begin{lem}\label{lem_composetrans}
	Let $f : M \to N$ and $g : N \to P$ be smooth maps. Let $g$ be transverse to a $\al C^1$-submanifold $S \subset P$ and $f$ be transverse to $g^{-1}(S)$. Then, $g \circ f$ is transverse to $S$.
\end{lem}

\begin{proof}
	If $(g\circ f)(M) \cap S = \emptyset$ then the result is trivial. The only interesting case is when there exists $x \in M$ such that $g \circ f(x) \in S$. Set $p = f(x)$. Since $g$ is transverse to $S$, $g^{-1}(S)$ is a submanifold of $N$ and 
	\begin{align}
	D_pg(T_pN) + T_{g(p)}S &= T_{g(p)} P. \label{1}
	\end{align}	
	Since $f$ is transverse to $g^{-1}(S)$ we have
	\begin{align}
	D_xf(T_xM) + T_{f(x)}(g^{-1}S) &= T_{f(x)}N. \label{2}
	\end{align}
	But $f(x) = p$, so (\ref{2}) becomes
	\begin{align}
	D_xf(T_xM) + T_{p}(g^{-1}S) &= T_{p}N. \label{3}
	\end{align}
	Substituting $T_pN$ from (\ref{3}) into  (\ref{1}) gives
	\begin{align*}
	D_pg(D_xf(T_xM) + T_p(g^{-1}(S)) +T_{g(p)} S &= T_{g(p)} P,
	\intertext{and by the chain rule}
	D_x(g \circ f)(T_xM) + D_{f(x)} g (T_{f(x)}(g^{-1}S)) + T_{g(f(x))}S &= T_{f(g(x))}P.
	\end{align*}
	Since $D_{f(x)} g(T_{f(x)}(g^{-1}S)) \subset T_{g(f(x))}S$ we get
	$$D_x(g \circ f)(T_xM) +  T_{g(f(x))}S = T_{f(g(x))}P.$$
This shows that $f \circ g$ is transverse to $S$ and the proof is complete.
\end{proof}

Next, we state Theorem 1.1 of Fischer \cite{fischer1}:

\begin{lem}\label{Fischer} Let $m \geq 0$ be an integer, $U \subset \bb R^n$ be a definable open set and $f : U \to \bb R$ be a definable $C^m$-function. Then, for every definable function $\varepsilon : U \to (0,\infty)$, there is a definable smooth function $g : U \to \bb R$ such that 
$$|\partial^{\alpha}f(u) - \partial^{\alpha}g(u)| < \varepsilon(u)$$
for any $u \in U$ and $|\alpha| \leq m$. 	
\end{lem}

We then have:

\begin{lem}\label{lem_appro_def} Let $U$ be an open definable subset of $\bb R^m$. Let $\varepsilon: U \to (0, \infty)$ be a continuous definable function. Then, for any $k \in \bb N$, there exists a smooth positive definable function $\varphi: U \to \bb R$ such that for all $\alpha \in \bb N^m$ with $|\alpha|\leq k$ we have
	$$ |\partial^\alpha \varphi (x)| < \varepsilon (x), \forall x \in U .$$
\end{lem}
\begin{proof}
By Lemma 1 in \cite{loi1} there is a definable positive $\al C^k$-function $\psi: U \to \bb R$ such that
\begin{equation*}
|\partial^\alpha \psi (x)| < \frac{\varepsilon (x)}{2},
\end{equation*}
$\forall x \in U$ and $|\alpha|\leq k .$
In particular, for $\alpha = 0$ we have:
\begin{equation}\label{eq12}
|\psi(x)| < \frac{\varepsilon (x)}{2}.
\end{equation}
Now, by Lemma \ref{Fischer} there is a smooth definable function $\varphi : U \to \bb R$ such that 
\begin{equation}\label{eq13}
 |\partial^\alpha \varphi (x) - \partial^\alpha \psi (x)|< \psi(x),
\end{equation}
$\forall x \in U$ and $|\alpha|\leq k.$
In particular, for $\alpha = 0$ we have:
$$|\varphi(x) - \psi(x)| < \psi(x),$$
for all $x \in U$. Since $\psi$ is positive, this implies that $\varphi$ is a positive definable function.
Then,
\begin{align*}
|\partial^{\alpha}\varphi(x)| & = |\partial^{\alpha}\varphi(x) - \partial^{\alpha} \psi(x) + \partial^{\alpha} \psi(x)| \\
& \leq  |\partial^{\alpha}\varphi(x) - \partial^{\alpha} \psi(x)| + |\partial^{\alpha} \psi(x)| \\
& <\psi(x) + \frac{\varepsilon(x)}{2}\\
&< \varepsilon(x) \tag{by (\ref{eq12}) and (\ref{eq13})}
\end{align*}
$\forall x \in U$ and $|\alpha|\leq k.$ This implies that $\varphi$ is the desired map and the proof is complete.
\end{proof}

Our next Lemma is a definable version of the basic transversality theorem whose proof is a standard application of Sard's theorem.  
\begin{lem}\label{lem_density_def} Let $M, S$ and $J$ be definable smooth manifolds, and $\Phi: M \times S \to J$ be a definable smooth submersion. Let $\Sigma$ be a finite collection of definable $C^1$ submanifold of $J$. Then 
	$$\tau (\Phi, \Sigma) = \{s \in S: \Phi_s = \Phi(., s) \pitchfork \Sigma\}$$ is a definable dense subset of $S$ with $\dim (S \setminus \tau (\Phi, \Sigma)) < \dim S$.
\end{lem}
\begin{proof}
	Similar to the proof of Lemma 3 in \cite{loi1}. 
\end{proof}

\begin{lem}\label{lem_a_open}Let $M $ and $N$ be definable smooth manifolds.  Let $\Sigma$ be a definable $\al C^1$ stratification of some  definable closed subset of $N$. If $\Sigma$ is Whitney $(a)$-regular then  
	$$\al A_{ \Sigma} := \{ f \in \al D^\infty (M, N):  f \pitchfork \Sigma\}$$ is open in $\al D^\infty(M, N)$ with  $\al D^\infty$-topology.
\end{lem}

\begin{proof}
	Notice that open in the $\al D^k$-topology implies open in the $\al D^p$-topology for $p\geq k$, so it is enough to prove the case $k=1$. But, this  follows from the same arguments used in Theorem 1 of Trotman \cite{trotman1} to prove a similar result in the smooth case (see also  \cite{trivedi}, \cite{trotman2}).
\end{proof}

Finally we prove the main theorem of this section.

\begin{proof}[Proof of Theorem \ref{thm_transversality}] 

(i) We assume that $M \subset \bb R^m$ and $N \subset \bb R^n$ are definable smooth submanifolds.  First we reduce the problem to the case when $M$ is an open set of $\bb R^m$ and $N=\bb R^n$.

Let $(T_M, \pi_M)$ and $(T_N, \pi_N)$ be definable tubular neighborhoods of $M$ and $N$ respectively. Let $\pi_{N*}: J_{\al D}^k(M, T_N) \to J_{\al D}^k(M, N)$ be the map given by $j^kf(x) \mapsto j^k (\pi_N \circ f)(x)$, where $f : M \to N \subset T_N$ is a smooth definable map. For each $S \in \Sigma$, set $S' := \pi_{N*}^{-1}(S)$. Since $\pi_{N*}$ is a definable submersion, $S'$ is a definable smooth manifold in $J_{\al D}^k(M, T_N)$. 

Let $f\in \al D^\infty(M, N)$. Consider $f$ as a map from $M$ to $\bb R^n$. Suppose that $g\in \al D^\infty(M, \bb R^n)$ is sufficiently close to $f$ such that $j^kg \pitchfork S'$. We may assume that $g \in \al D^\infty(M, T_N)$. Note that $j^k(\pi_N \circ g) = \pi_{N*} \circ j^k g$. Then by Lemma \ref{lem_composetrans} we have that  $j^k(\pi_N \circ g) \pitchfork S$. Moreover, since the map
$\al D^\infty(M, T_N) \to \al D^\infty(M, N)$ given by $g \mapsto \pi_N \circ g$ is continuous by Lemma \ref{lem_jet_map2}, it follows that $\pi_N\circ g$ is sufficiently close to $f$. Therefore, we can reduce the proof to the case $N= \bb R^n$.

For each $S \in \Sigma$, set
$$S'' := \{(x, y, a_\alpha)_{1\leq |\alpha|\leq k} \in J_{\al D}^k(T_M, \bb R^n): (\pi_M(x), y, a_\alpha)_{1\leq |\alpha|\leq k} \in S\}.$$
Clearly, $S''$ is a definable smooth submanifold of 
$J_{\al D}^k(T_M, \bb R^n)$. 

Let $f\in \al D^\infty(M, \bb R^n)$. Set $\tilde{f} := f \circ \pi_M:  T_M \to \bb R^n$. Suppose $g \in  \al D^\infty(T_M, \bb R^n)$ is sufficiently close to $\tilde{f}$ such that $j^kg\pitchfork S''$. By Lemma \ref{lem_restriction}, $g \circ \pi_M$ is also close to $f$, and again as a consequence of Lemma  \ref{lem_composetrans},  $j^k (g \circ \pi_M) \pitchfork S$.  Thus we can reduce the proof to the case when $M$ is an open set of $\bb R^m$. 

Let $f \in \al D^\infty( M, \bb R^m)$ and $U_f$ be an open neighborhood of $f$ in the $\al D^\infty$-topology. It suffices to show that there is $g \in  U_f$ such that $j^kg \pitchfork \Sigma$. 

Since $U_f$ is open in the $\al D^\infty$-topology, there are $l \in \bb N$ and a definable continuous function $\varepsilon: M \to (0,\infty)$ such that $\al U_\varepsilon^l(f) \subset U_f$.  We can choose $l$ such that $l \geq k$ (since as $l$ increases, $\al U_\varepsilon^l(f)$ gets smaller). We will show that $\al U_\varepsilon^l(f)$ contains an element $g$ such that $j^k g \pitchfork \Sigma$.  We separate the proof into two cases as follows.

\textbf{Case 1:} $M$ is an open subset of $\bb R^m$ and $N = \bb R$. 

Let $R = \#\{\alpha \in \bb N^m: |\alpha|\leq l\}$, $C= R^2 (l!)^2$. 

By Lemma \ref{lem_appro_def}, there is a definable smooth positive function $\varphi: M \to \bb R$ such that 
$$|\partial^\alpha \varphi(x)| < \frac{\varepsilon(x)}{C(1 +\|x\|^l)}, \text{ for all } |\alpha| \leq l.$$

Consider the following family of maps
$$ F(x, t)= f_s(x) := f(x) + \sum_{|\alpha|\leq k} s_\alpha x^\alpha\varphi(x),$$
where $x \in M$ and $s = (s_\alpha)_{|\alpha|\leq k}\in I^R$ with $I = (0,1)$. 
By the same computation as in Theorem 2 of \cite{loi1}, it follows that $f_s \in \al U_\varepsilon^l (f)$ for all $s\in I ^R$. In addition the map 
$$\Phi: M \times I^R \to J^k_{\al D}(M, \bb R), (x, t) \mapsto(x, (\partial^\alpha f_s(x)))_{|\alpha|\leq k}$$ is a submersion.  By Lemma \ref{lem_density_def}, there exists an $s$ such that $j^k f_s \pitchfork \Sigma$. Then, $g = f_s$ has the required property. 

\textbf{Case 2:} $M$ is an open subset of $\bb R^m$ and $N = \bb R^n$.  The result is obtained by applying the same arguments as in the case 1 for the family 
$$F(x, s) := (f_1(x) + \sum_{|\alpha|\leq k} s_{1,\alpha} x^\alpha\varphi(x),\ldots, f_n(x) + \sum_{|\alpha|\leq k} s_{n,\alpha} x^\alpha\varphi(x)),$$
where $s = (s_{i, \alpha})_{1\leq n, |\alpha|\leq k} \in I^{nR}$.

(ii) Consider the map
$$ j^k: \al D^\infty(M, N) \to \al D^\infty(M, J^k_{\al D} (M, N)).$$ Let $\al L:= \{ f\in \al D^\infty(M, J^k_{\al D} (M, N)), f \pitchfork \Sigma\}$. Since $\Sigma$ is an  (a)-regular stratification, $\al L$ is open in $\al D^\infty(M, J^k_{\al D} (M, N)$ (see Lemma \ref{lem_a_open}). By Lemma \ref{lem_jet_map1}, this shows that the map $j^k$ is continuous, therefore, $\al A^k_\Sigma = (j^k)^{-1} (\al L)$ is open.
\end{proof}

\section{The set of non-transverse maps}

Let $M$ and $N $ be definable smooth manifolds. Let $\Sigma$ and $\Sigma'$ be finite collections of $\al C^1$-submanifolds of $N$. We write $\Sigma < \Sigma'$ if elements of $\Sigma$ are also elements of $\Sigma'$. If $\Sigma$ and $\Sigma'$ are stratifications of the same set $V \subset N$ then $\Sigma'$ is said to be a refinement of $\Sigma$ if every stratum of $\Sigma$ is a union of some strata of $\Sigma'$. 

Denote by $\al A_{\Sigma}$ the set of maps transverse to $\Sigma$, i.e.
$\al A_{\Sigma} := \{f \in \al D^{\infty}(M,N) : f \pitchfork \Sigma\}.$ The following observations are clear from the definitions:

(1) If $\Sigma < \Sigma'$ then $\al A_{\Sigma'} \subset \al A_{\Sigma}$.

(2) If $\Sigma'$ is a refinement of $\Sigma$ then $\al A_{\Sigma'} \subset \al A_{\Sigma}$.

We will now prove the main theorem of this section.

\begin{thm}\label{cor_strong_transversality}
 Let $\Sigma$ be a finite collection of definable $\al C^1$-submanifolds of $J_{\al D}^k(M, N)$. For $k\in \bb N$, define 
	$$\al A^k_{ \Sigma}:=\{ f \in \al D^\infty (M, N): j^k f \pitchfork \Sigma\}.$$ Then $\al D^\infty(M, N) \setminus \al A^k_{ \Sigma}$ is nowhere dense in $\al D^\infty(M, N)$ with the $\al D^\infty$-topology.
\end{thm}

\begin{proof} It suffices to show that $\al A^k_\Sigma$ contains a subset which is open and dense in $\al D^\infty(M, N)$ with the  $\al D^\infty$-topology.
	
Let $X$ be the union of all elements of $\Sigma$. We denote by $\overline{X}$ the closure of $X$ in $J_{\al D}^k(M, N)$. Let $\Sigma'$ be a definable $(a)$-regular stratification of $\overline{X}$ of $X$ which is compatible with the family consisting of  $\overline{X}\setminus X$ and all strata of $\Sigma$ (this is possible due to \cite{loi3}, \cite{loi2}  or \cite{ntt}). Write 
	$$ \Sigma' = \{\Sigma'_1, \Sigma'_2\}$$ where 
	$\Sigma'_1$ is the stratification of $\overline{X}\setminus X$ and $ \Sigma'_2$ is a refinement of $\Sigma$. By the observations above, that shows
	$$\al A^k_{\Sigma'} \subset \al A^k_{\Sigma'_2} \subset \al A^k_{\Sigma}.$$
	Since $\Sigma'$ is a Whitney $(a)$-regular stratification of the closed set $\overline{X}$, by Theorem \ref{thm_transversality}, $\al A^k_{\Sigma'}$ is open and dense  in $\al D^\infty(M, N)$ with the $\al D^\infty$-topology. This ends the proof. 
\end{proof}

\begin{cor}\label{cor_non_transverse}
    Let $\Sigma$ be a finite collection of definable $\al C^1$-submanifolds of $N$. Then $\al D^\infty(M, N)\setminus  \al A_\Sigma$ is nowhere dense in $\al D^\infty(M, N)$ with the $\al D^\infty$-topology. 
\end{cor}
\begin{proof} Note that $J^0_{\al D}(M, N) = M \times N$ and $j^0 f(x) = (x, f(x))$. Since the projection $\pi: M \times N \to N$ $(x, y) \mapsto y$ is a submersion, $\Sigma':= \pi^{-1}(\Sigma)$ is a collection submanifolds of $M \times N$.  By Lemma \ref{lem_composetrans}, if $f \in \al D^\infty (M, N)$ such that $j^0 f \pitchfork \Sigma'$ then $ f\pitchfork \Sigma$. This implies that the set 
	$\al A^0_{\Sigma'} := \{f \in \al D^\infty(M, N): j^0f\pitchfork \pi^{-1}(\Sigma)\}$ is contained in $\al A_\Sigma$. By Theorem \ref{cor_strong_transversality}, the complement of $\al A^0_{\Sigma'}$ in $\al D^\infty(M, N)$ is nowhere dense, then so is the complement of $\al A_\Sigma$. 
\end{proof}

\section{Openness implies Whitney (a)-regularity}

Let $\Sigma$ be a definable stratification of a definable set $V\subset \bb R^n$. Suppose that $p \in V$ lies in a stratum $S$ of $\Sigma$. Consider a linear subspace $H$ of $T_p\bb R^n \equiv \bb R^n$ such that $H + T_p S= \bb R^n$. Let $M$ be a definable smooth manifold such that $\dim M \geq \dim H$. We say that a map $f : M \to \bb R^n$ \emph{has a given derivative} at a point $x \in M$ if:

1. $f(x) = p$,

2. $D_xf(T_xM) = H$.

Set $\fs D_H := \{f \in \al D^{\infty}(M,\bb R^n) :\exists\, x \in M,  f(x) = p, D_xf(T_xM) = H\}.$ We will show that there exists a map $f \in \fs D_H$ transverse to $\Sigma$.

Notice that the transversality theorem guarantees the existence of maps transverse to any stratification but not of transverse maps with a given derivative. For example, if the dimension of $M$ is less than the codimension of a submanifold $S$ of $\bb R^n$, then the only maps transverse to $S$ are those that do not intersect $S$.

\begin{lem}\label{givendf} There exists an $f \in \fs D_H$ such that $f \pitchfork \Sigma$.
\end{lem}

\begin{proof} Without loss of generality we assume that $H= \bb R^k \times \{0\} \subset \bb R^n$ and $p = 0$. First we will show that there exists a definable smooth map $g : \bb R^k \to \bb R^n$ such that $g(0) = 0$, $D_0g(\bb R^k) = H$ and $g \pitchfork \Sigma$.  
	
Consider the map $\Psi : \bb R^{k}\times\bb R^{n-k} \to \bb R^n$ given by
$$\Psi(x,s)= \psi_s(x) := (x_1,\ldots,x_k,s_1\|x\|^2, \ldots, s_{n-k}\|x\|^2)$$
where $x = (x_1, \ldots, x_k)$, $s = (s_1, \ldots, s_{n-k})$ and $\|x\|^2 = x_1^2+\cdots+x_k^2$. The Jacobian of this map looks like
\begin{align*}
\left [
\begin{array}{c | c}
I_k & 0\\
\hline
\begin{array}{ccc}
. & . & .\\
. & . & .\\
. & . & .\\
\end{array}
& 
\begin{array}{ccc}
{\scriptstyle \|x\|^2} &  \ldots & 0 \\
\vdots & \ddots & \vdots \\
0 & \ldots & {\scriptstyle \|x\|^2}\\
\end{array}
\end{array}\right ]_.
\end{align*}
It is then clear that $\Psi$ is a submersion on $\{\bb R^k \setminus \{0\}\} \times \bb R^{n-k}$. Put $\Psi' := \Psi|_{\{\bb R^k \setminus \{0\}\} \times \bb R^{n-k}}$. By Lemma \ref{lem_density_def},
 $$\al S := \{s \in \bb R^{n-k} : \psi'_s := \Psi'(x,s) \pitchfork \Sigma\}$$
is dense in $\bb R^{n-k}$. We claim that $\forall s \in \al S$, $\psi_s \pitchfork \Sigma$. Indeed, at $x = 0$ we have $D_0\psi_s (\bb R^k) = H$ which is transverse to $\Sigma$ at $\psi_s(0)=0$, so $\psi_s \pitchfork_0 \Sigma$. For $x \neq 0$, we have $\psi_s = \psi_s'$ which is transverse to $\Sigma$. This proves the claim. 

Next, we fix an $s \in \al S$ and show that there exists a definable smooth map $h : M \to \bb R^k$ such that $h(x) = 0$ for some $x \in M$ which is a submersion at $x$ and $h \pitchfork \psi_s^{-1}(\Sigma)$. Take a point $x \in M$ and a definable smooth coordinate chart $\xi : V \to \bb R^m$ around $x$ such that $\xi(x) = 0$. Now, define $L : \bb R^m \to \bb R^k$ by 
$$L(x_1,\ldots,x_m) = (x_1,\ldots,x_k).$$ 
Since $m \geq k$, $L$ is a linear submersion at $0 \in \bb R^m$. Put $h' := L\circ \xi : V \to \bb R^k$. Since $\xi^{-1}$ is a diffeomorphism, $h'$ is a submersion at $x$. We can then smoothly extend $h'$ to $M$ by means of a smooth definable bump function. Let us call this extension $\tilde{h}$ and notice that it is a submersion at $x$.

Consider the set $\fs S$ of definable smooth maps between $M$ and $\bb R^k$ that are submersion at $x \in M$. Obviously, $\fs S$ is non-empty since $\tilde{h}\in \fs S$. Moreover, it is easy to see that $\fs S$ is an open set in $\al D^{\infty}(M,\bb R^n)$ with the $\al D^\infty$-topology. Since the set of smooth definabe  maps transverse to $\psi_s^{-1}(\Sigma)$, say $\fs T$, is dense in $\al D^{\infty}(M,\bb R^n)$ with the $\al D^\infty$-topology (by Corollary \ref{cor_non_transverse}) the intersection of $\fs S$ with $\fs T$ is non-empty. Take $h \in \fs S \cap \fs T$. Then $h$ is a smooth definable map  transverse to $\psi_s^{-1}(\Sigma)$ which is also a submersion at $x \in M$.

Now, put $f := \psi_s \circ h : M \to \bb R^n$. We claim that $f$ is the required map, i.e. $D_xf(T_xM) = H$ and $f$ is transverse to $\Sigma$. Just following the construction of $f$ it is easy to see that $D_xf(T_xM) = H$. We need to verify that $f$ is transverse to $\Sigma$. This fact immediately follows from Lemma \ref{lem_composetrans} and the proof is complete.
\end{proof}

We can now prove a definable version of the main result of Trotman \cite{trotman1} (see also \cite{trotman2}). The result is stated as follows. 

\begin{thm}\label{thm_openness_a_regularity} Let $N$ be a definable smooth manifold. Let $\Sigma$ be a definable stratification of some definable closed subset of $N$. Then the following statements are equivalent:
	
	(1) $\Sigma$ is Whitney $(a)$-regular.
	
	(2) For any definable smooth manifold $M$, the set $\{f \in \al D^{\infty}(M,N) : f \pitchfork \Sigma\}$ is open in $\al D^{\infty}(M,N)$ with the $\al D^\infty$-topology.
	
	(3) For any definable smooth manifold $M$ of dimension greater than or equal to the codimension (in $N$) of the smallest stratum in $\Sigma$, the set 
	$$\{f \in \al D^\infty(M, N): f \pitchfork \Sigma\}$$ is open in $\al D^\infty(M, N)$ with the $\al D^\infty$-topology.
\end{thm}

\begin{proof} 
	That (1) implies (2) follows from \ref{lem_a_open}. That (2) implies (3) is trivial. 
	
	We show $(3)$ implies $(1)$. Suppose $\Sigma$ is not Whitney $(a)$-regular. Then there exist a pair of adjacent strata $(X, Y)$ of $\Sigma$ and a point $y\in Y$ such that $X$ is not $(a)$-regular over $Y$ at $y$. Let $(U, \phi)$ be a  definable smooth chart of $N$ around the point $y$ such that $\phi(U) = \bb R^n$ and $\phi(y) = 0$ where $n = \dim N$. Set $\Sigma' := \phi (\Sigma' \cap U)$. Since diffeomorphisms preserve\footnote{In fact, $(a)$-regularity is a $C^1$-invariant; see Proposition 2.1.2 in the Ph.D. thesis \cite{Nhan2} of the first author.} Whitney $(a)$-regularity, $\Sigma'$ is not $(a)$-regular. More precisely, if $X':= \phi (X \cap U)$ and $Y' := \phi(Y \cap U)$, then $X'$ is not $(a)$-regular over $Y'$ at $0$. This means that there is a sequence $\{x_i\}$ in $X'$ tending to $0$ such that $T_{x_i} X'$ tends to $\tau \in \bb G_n^{l}$ ($l =\dim X'$) and $T_0 Y' \not\subset \tau$.
	
Let $r = \dim Y'$. By the classical arguments as in \cite{trotman1}, there is a linear subspace $H$ of $\bb R^n$ with $\dim H = n -r$ such that 
	$H + T_0 Y' = \bb R^n$ and $H +\tau \neq \bb R^n$. Suppose that $\dim (H +\tau) = b$ and $\dim (H \cap \tau) = c$. Note that $c \leq n-r\leq b < n$.  Let $\{v_1, \ldots,v_n\}$ be basis of $\bb R^n$ such that $H$ is spanned by $\{v_1, \ldots, v_{n-r}\}$ and $\tau$ is spanned by $\{v_{n-r-c}, \ldots, v_{b}\}$.
	
	Now, at each point $x_i \in X'$, choose $\{v^i_1,  \ldots v^i_n\}$ to be a basis of $T_{x_i} \bb R^n$ such that $T_{x_i}X'$ spanned by $\{v^i_{n-r-c}, \ldots, v^i_{b}\}$ and $\lim_{i\to \infty} v^i_j= v_j$ for $j = 1, \ldots, n$. Denote by $H_i$ the linear subspace spanned by $\{v^i_1, \ldots, v^i_{n-r}\}$. We have 
	\begin{equation*}\label{eq_linear_spaces}
	\lim_{i\to \infty}H_i = H \hspace{1cm} \text{and}\hspace{1cm} H_i + T_{x_i}X' \neq T_{x_i} \bb R^n =\bb R^n. \tag{$\dagger$}
	\end{equation*}
	
	Next, for each $i$ we denote by $A_i: \bb R^n \to \bb R^n$  the linear isomorphism such that $A_i (v_j) = v_j^i, \forall j = 1\ldots, n$. Note that $A_i$ is uniquely determined.  Since $\lim_{i \to \infty} v_j^i = v_j,  \forall j = 1\ldots, n$, $\lim_{i \to \infty} A_i = {Id}_n$. 
	
Let $M$ be a definable smooth manifold as in the hypothesis. Note that $\dim M \geq \dim H$  since $\dim M \geq n- \min_{S\in \Sigma}\{\dim S\}$. By Lemma \ref{givendf}, there exists a definable smooth map $f: M \to \bb R^n$ such that $f(x) =0$ (for some $x\in M$), $D_x f (T_xM) = H$ and $f\pitchfork \Sigma'$.

	Let $f_i :=  A_i\circ f + x_i$.  We have $f_i(x) = x_i$, $D_x f_i (T_x M) = H_i$ but $f_i$ does not converge to $f$ in the $\al D^\infty$-topology since it is not `close' to $f$ at infinity.  Now let $\lambda: M \to [0, 1]$ be a definable smooth function whose value equals $1$ in a small definable neighborhood of $x$ and equals $0$ outside a bigger definable compact neighborhood of $x$ ($\lambda$ can be constructed by using exponential bump functions). Define $g_i := \lambda f_i + (1-\lambda) f$. It is easy to check that 
	$g_i(x) = x_i$, $D_x g_i (T_x M) = H_i$
	and  $\{g_i\}$ converges to $f$ in $\al D^\infty$-topology. It follows from (\ref{eq_linear_spaces}) that  $f_i$ is not transverse to $\Sigma'$ at $x$ while $f$ is transverse to $\Sigma$ at every point on $M$. Set $h:= \phi^{-1}\circ f$ and $h_i =\phi^{-1} \circ g_i$. By Lemma \ref{lem_composetrans}, $h_i \not \pitchfork \Sigma$ and $ h\pitchfork \Sigma$. On the other hand, since $\phi$ is a definable smooth diffeomorphism and $\{f_i\}$ tends to $f$, $\{h_i\}$ tends to $h$ (in the $\al D^\infty$-topology). This shows that the set 
	$$\{f \in \al D^\infty(M, N): f \pitchfork \Sigma\}$$ 
	is not open, a contradiction. 
\end{proof}

\begin{rem} 
(i) The results of Theorem \ref{cor_strong_transversality}  and Theorem \ref{thm_openness_a_regularity} hold for the $\al D^k$-topology ($0<k < \infty$) without any restriction on the o-minimal structure.

(ii) The method used to prove the above theorem works also in the smooth case. Thus we  get a new more constructive proof of the theorem of Trotman (see \cite{trotman2}, \cite{trotman1}).
\end{rem}

\end{document}